\documentclass[10pt,twoside]{amsart}
\usepackage{mathtools,amsmath,amssymb}
\usepackage{mathrsfs}

\usepackage{times}           
\usepackage[a4paper,top=1in, bottom=1in, left=1in, right=1in]{geometry} 
\usepackage{graphicx} 
\usepackage[colorlinks=true,linkcolor=blue,citecolor=black]{hyperref} 
\usepackage{cleveref}
\usepackage[backend=biber,backref=true,doi=false,url=false,maxnames=9,isbn=false]{biblatex} 
\AtEveryBibitem{\clearfield{pagination}}
\AtEveryBibitem{\clearfield{pages}}
\renewbibmacro*{pageref}{%
 \iflistundef{pageref}
   {}
   {\addspace
    \mkbibbrackets{\printlist[pageref][-\value{listtotal}]{pageref}}}} 
\usepackage{tikz-cd}
\usepackage{mathptmx}
\usepackage{float}
\usepackage{amssymb}
\usepackage[english]{babel}
\usepackage{amsthm}
\usepackage{amsmath}
\usepackage{mathtools}
\usepackage{soul}
\usepackage{array}   
\usepackage{float}
\usepackage{orcidlink}
\usepackage{nomencl}

\makenomenclature
\usepackage[toc,page]{appendix}
\usepackage{caption}
\usepackage{mathrsfs}     
\usepackage{csquotes}
\newcommand{\GL}{\mathrm{GL}}
\newcommand{\M}{\mathrm{M}}

\newcommand{\F}{\mathscr{F}}
\newcommand{\J}{\mathscr{J}}

\newcommand{\C}{\mathbb C}
\newcommand{\und}[1]{\underline{#1}}
\newcommand{\diag}{\mathrm{diag}}
\newcommand{\D}{\mathrm{D}}
\usepackage{thmtools, thm-restate}

\newtheorem{lemma}{Lemma}[section]
\newtheorem{proposition}[lemma]{Proposition}
\newtheorem{corollary}[lemma]{Corollary}
\theoremstyle{definition} 
\newtheorem{definition}[lemma]{Definition}
\newtheorem{example}[lemma]{Example}
\newtheorem{remark}[lemma]{Remark}

\title[Dynamics of Word Maps and Polynomial Maps]{Dynamics of Word Maps on Groups and Polynomial Maps on Algebras}
\date{\today}
\author[Saikat Panja]{Saikat Panja\orcidlink{0000-0002-9639-3122}}
\email{panjasaikat300@gmail.com}
\address{Indian Statistical Institute, Bengaluru Centre, 8th Mile, Mysore Rd, RVCE Post, Gnana Bharathi, Bengaluru, Karnataka 560059, India}
\thanks{Panja is supported by an NBHM postdoctoral fellowship, file number ending at R\&D-II/6746.}
\date{\today}
\begin{document}
\begin{abstract}
We introduce the notions of Fatou and Julia sets in the context of word maps on complex Lie groups and polynomial maps on finite-dimensional associative $\mathbb C$-algebras.
For the group-theoretic question, we investigate the dynamics of the power map $x \mapsto x^{M}$ on the Lie group $\mathrm{GL}_n(\mathbb C)$, where $M \geq 2$ is an integer.
For the algebra-related question, we study polynomial self-maps of $\mathrm{M}_n(\mathbb C)$ induced by monic polynomials in one variable. 
In both cases, we pin down the explicit description of the Fatou and Julia sets.
We also show that there does not exist any wandering Fatou component of the pair $(p,\M_n(\C))$ where $p\in\C[z]$ is a monic polynomial of degree $\geq 2$.
\end{abstract}
\maketitle
\section{Introduction}\label{sec:introin}
\subsection{Complex dynamics} The modern study of discrete complex dynamical systems dates back to the early twentieth century, beginning with the seminal work of P.\ Fatou and G.\ Julia on the iteration of rational maps $\varphi(z)\in\mathbb{C}(z)$ on the Riemann sphere; see, for instance, \cite{GastonJulia1918,Fatou1926}. 
A natural point of departure is the analysis of fixed points—those $z\in\widehat{\C}$ satisfying $\varphi(z)=z$. These are classified by the value of $\varphi'(z)$ as attracting, repelling, or neutral, reflecting whether nearby points approach or recede from the fixed point under iteration; periodic and preperiodic points are also studied in the same spirit.

More subtle features of the dynamics appear in the Fatou and Julia sets. The Fatou set consists of points where the iterates of $\varphi$ vary stably, in the sense that small perturbations remain small under iteration. Its complement, the Julia set, is the locus of chaotic behaviour, where arbitrarily small errors can be amplified indefinitely.
Towards the end of the twentieth century, this theory was extended to several complex variables in foundational work by Bedford--Smillie, Forn\ae ss--Sibony, and others; see \cite{BedfordSmillie1991a,BedfordSmillie1991b,ForSib1994a,ForSib1995}. 
\subsection{Word maps on groups and polynomial maps on algebras}
Let $\mathbf{C}$ be the category of groups (or of $k$-algebras for some field $k$). 
Let $w \in \mathbf{F}_r$, where $\mathbf{F}_r \in \mathbf{C}$ is the free group (resp.\ free algebra) on $r$ generators. 
For any object $X \in \mathbf{C}$, one obtains a map
\begin{align*}
    \widetilde{w} : X^{r} \longrightarrow X,
\end{align*}
defined by substitution. 
When $X$ is a group (resp.\ an algebra), the map $\widetilde{w}$ is called a \emph{word map} (resp.\ a \emph{polynomial map}). 
By abuse of notation, we denote the map$\widetilde{w}$ by $w$ itself.

The study of word maps dates back at least to 1951, to work of {\O}ystein Ore, who proved that every element of the alternating group $A_n$ is a commutator and conjectured that the same holds for all finite non-abelian simple groups; see \cite{Ore1951}. 
This conjecture was resolved in full generality through the efforts of many mathematicians; see \cite{LiebeckObrienShalevTiep2010} and the references therein. 
In a related direction, a striking result was later established: for every finite non-abelian simple group $G$ of sufficiently large order, and every nontrivial word $w$, one has $w(G)^2 = G$; see \cite{LarsenShalevTiep2011}.
However $2$ can not be reduced further; for example squaring map on $A_5$ is not surjective.
In this article we study the \emph{power map} $x \mapsto x^{M}$ for an integer $M \ge 2$. 
Such maps have been investigated extensively: by Chatterjee and Steinberg, independently, in the context of algebraic groups \cite{Chatterjee02,Steinberg03}; for some finite groups of Lie type in \cite{KunduSingh2024, PanjaSingh2025}, for general linear groups over finite principal ideal local rings of length two in \cite{PanjaRoySingh2025}; see also \cite{Panja2025roots} and the references therein. 
Power maps exhibit a number of interesting features (see \cite{Panja2024c}), and they also provide a useful tool for studying the number of real conjugacy classes \cite{panja2024d}.

On the other hand, the study of polynomial maps has a long history. 
For example, Shoda proved in 1937 that over a field of characteristic zero, every trace-zero matrix is a commutator; see \cite{Shoda1937}. 
Kaplansky and L\'vov later conjectured that the image of a multilinear polynomial on $\mathrm{M}_n(k)$, the full matrix algebra over an infinite field $k$, is always a vector space. 
Several positive results supporting this conjecture are known--see, for instance, \cite{KanelMalevRowen2012,KanelMalevRowen2016}--but the conjecture remains wide open in the cases $n \geq 4$. 
Br\v esar in 2020 proved that \cite{Bresar2020} if $C$ is a commutative unital algebra over a field $\mathbb F$ of characteristic $0$, $A=\M_n(C)$, and the polynomial $f$ is neither an identity nor a central polynomial of $A$, then every commutator in $A$ can be written as a difference of two elements, each of which is a sum of $7788$ elements from $f(A)$; for a stronger result see \cite{BresarVolcic2025}. 
Polynomial maps have also been investigated in other algebraic settings: for instance, upper triangular matrix algebras were studied in \cite{PanjaPrasad2023}, and Octonion algebras in \cite{PanjaSainiSingh2025}; see also \cite{PanjaSainiSinghConst} for a more general formulation of this problem.
\subsection{Motivation and the statement of the result}
Motivated by the study of several complex variables, we aim to study Fatou and Julia sets in the context of word maps on complex Lie groups and polynomial maps on $\C$-algebras.
Since a word map is a map $w: X^r\longrightarrow X$, it does not make sense to iterate; that is why we need to work with a tuple of maps, see \Cref{sec:prelim}. 
However, when one considers the free object on one generator $\mathbf F_1\in\mathbf C$, then a relevant $w\in \mathbf F_1$ induces a map $w: X\longrightarrow X$. 
With this in mind, we investigate the dynamics for the power map on the group $\GL_n(\C)$ and the polynomial map in one variable on the algebra $\M_n(\C)$. Our main result is the following:
\begin{restatable}{theorem}{theoremone}\label{thm:FJ-polynomial-map}
Let $p\in\C[x]$ be a monic polynomial of degree $\geq 2$ and let
$X\in\M_n(\C)$. Then $X\in\F_p(\M_n(\C))$ if and only if $\sigma(X)\subseteq \F_p(\C)$, where $\sigma(X)$ denotes the spectrum of $X$. Furthermore $\F_{x^M}(\GL_n(\C))=\F_{x^M}(\M_n(\C))\cap\GL_n(\C)$.
\end{restatable}
This can be seen as a generalisation of results obtained in \cite{Pal2019}.
After providing necessary definitions and discussing some examples in \Cref{sec:prelim}, we prove the main theorem in \Cref{sec:ring-poly}.

\section{Preliminary definitions and known results}\label{sec:prelim}
\begin{definition}
    Let $G$ be a complex Lie group and $w\in \F_r$ be an element from the free group on $r$ generators, say $x_1,x_2,\ldots,x_r$. 
    This defines a \emph{word map on $G$},
    $(g_1,g_2,\cdots,g_r)\mapsto w(g_1,g_2,\cdots,g_r)$, by plugging $g_i$ in place of $x_i$.
    Given $r$ words $w_1,\cdots,w_r$ one gets a map $\underline{w}:G^r\longrightarrow G^r$ by defining: $$\underline{g}=(g_1,g_2,\cdots,g_r)\mapsto({w_1(\underline{g}),\cdots,w_r(\underline{g})}).$$
    The \emph{Fatou set of the pair $(G,\underline{w}=(w_1,\ldots,w_r))$} consists of the set of points $x\in G^r$ which have a neighborhood $U$ such that $(\underline{w}^n|_U)_n$ is normal. 
    Here $\underline{w}^1(x)=\underline{w}(x)$, $\underline{w}^2(x)=\underline{w}(\underline{w}(x))$ and more generally $\underline{w}^n(x)=\underline{w}(\underline{w}^{n-1}(x))$.
    
    With necessary modification, for a topological $\C$-algebra $A$ and $f_1,\cdots,f_r\in\C\langle x_1,x_2,\ldots,x_r\rangle$ the \emph{Fatou set of the pair $(A,\underline{f}=(f_1,\cdots,f_r))$} consists of the set of points $x\in A^r$ which have a neighborhood $U$ such that $(\underline{f}^n|_U)_n$ is normal.
    
    The \emph{Julia set} of the pair is defined to be the set-theoretic complement of the Fatou set.
\end{definition}
The Julia set for the pair $(G,w)$, of a group and a word, is denoted by $\J_w(G)$, and the Fatou set is denoted by $\F_w(G)$; similar notation is followed for the pair $(A,f)$, of an algebra and a polynomial.
Our definitions are motivated by the theory of several complex variables.
In this article, we compute the Fatou set (consequently Julia set) for the pairs $(\GL_n(\mathbb C),x^M)$ and $(\M_n(\mathbb C), x^M+\sum\limits_{i=0}^{M-1}a_ix^i)$, where $\GL_n(\C)$ is considered to be a topological group and $\M_n(\C)$ is considered to be a topological algebra; here $a_i\in \C$ for all $i=1,2,\ldots, M-1$.

In the spirit of the theory of a single complex variable, we define the \emph{filled Julia set} for the pair $(G,\und{w})$ to be 
\begin{align*}
    K_{\und{w}}(G)=\left\{x\in G^r:\left(\und{w}^n(x)\right)_{n\geq 0}\text{ is a bounded sequence}\right\}.
\end{align*}
For a polynomial function (of one variable) $p$ on $\C$, one has the Julia set $J_p(\C)=\partial K_p(\C)$; see \cite[Lemma 9.4]{Milnor2006book}.
However, this need not continue to hold in general, as we see in later sections.
So one needs to be careful while dealing with the case of several variables.
\begin{definition}
    Let $G$ be a complex Lie group (resp. $A$ be a $\C$-algebra). Given a word map (resp. polynomial map) $\chi$, the associated Fatou set is the disjoint union of its connected components, known as \emph{Fatou components}.
    The Fatou components are permuted by $\chi$.
    For a Fatou component $U\subseteq G^r$ (resp. $\subseteq A^r$), it is (a) \emph{periodic} if $\chi^p(U)=U$ for some $p>0$, (b) \emph{preperiodic} if $\chi^k(U)$ is periodic, and (c) \emph{wandering Fatou component} if the sets $\left\{\chi^k(U)\right\}_{k\geq 0}$ are pairwise disjoint.
\end{definition}
Let us now see some examples. 
\begin{example}
Consider the group $\mathrm{GL}_{2}(\mathbb{C})$ with $r=2$.  
Take the words $w_{1}=x_{2}$ and $w_{2}=x_{1}^{2}x_{2}\in \mathbf{F}_{2}$, and the pair of matrices
\begin{align*}
(A,B)=\left(
\begin{pmatrix}1 & 1\\ 0 & 1\end{pmatrix},
\begin{pmatrix}1 & 0\\ 1 & 1\end{pmatrix}
\right).
\end{align*}
Then the induced map is
\begin{align*}
W(X,Y)=(Y,\, X^{2}Y).
\end{align*}
For instance,
\begin{align*}
W^{3}(A,B)=
\left(
\begin{pmatrix}3 & 2\\ 7 & 5\end{pmatrix},
\begin{pmatrix}89 & 62\\ 33 & 23\end{pmatrix}
\right).
\end{align*}
Moreover,
\begin{align*}
W^{6}(A,B) &=\left(
\begin{pmatrix}
69210849 & 48219134\\
25665025 & 17880799
\end{pmatrix},
\begin{pmatrix}
1557268252466751 & 1084947340259330\\
3792477575677951 & 2642215592726081
\end{pmatrix}\right).
\end{align*}
\end{example}

\begin{example}
Consider the algebra $\mathrm{M}_{2}(\mathbb{C})$ with $r=2$.  
Let $w_{1}=x_{2}$ and $w_{2}=x_{1}^{2}+x_{1}+x_{2}\in \mathbf{F}_{2}$, and take
\begin{align*}
(A,B)=\left(
\begin{pmatrix}1 & 0\\ 1 & 1\end{pmatrix},
\begin{pmatrix}1 & 1\\ 0 & 1\end{pmatrix}
\right).
\end{align*}
Then
\begin{align*}
W^{3}(A,B)=
\left(
\begin{pmatrix}5 & 4\\ 3 & 5\end{pmatrix},
\begin{pmatrix}20 & 11\\ 24 & 20\end{pmatrix}
\right),\\
W^{5}(A,B)=
\left(
\begin{pmatrix}62 & 55\\ 57 & 62\end{pmatrix},
\begin{pmatrix}746 & 506\\ 1041 & 746\end{pmatrix}
\right).
\end{align*}
\end{example}

For later use, we tally the Sullivan classification of Fatou Components on the Riemann sphere.
\begin{lemma}\cite[Theorem 16.1]{Milnor2006book}\label{lem:suli-fatou-comp}
Let $f$ be a rational function on the Riemann sphere. 
If $f$ maps the Fatou component $U$ onto itself,
then there are just four possibilities, as follows: Either $U$ is the
immediate basin for an attracting fixed point or for one petal of a parabolic fixed point which has multiplier $\lambda=1$ or else $U$ is a Siegel disk or Herman ring.
\end{lemma}
Further by \cite[Lemma 9.4]{Milnor2006book}, all the bounded components of the Fatou set of a polynomial of degree $\geq 2$ are simply connected. 
Thus while analyzing the elements of the Fatou components of the polynomial $f$, one gets exactly three possibilities; Herman ring does not occur (as it is conformally isomorphic to some annulus).
To check whether the family $\{w^m\}_{m\geq0}$ is normal or not, we use the criterion developed by Gerado and Krantz which applies to complex manifold. 
For future reference we note it down here.
\begin{lemma}\cite[Theorem 3.1]{GeradoKrantz1991}\label{lem:normal-GK}
Let $\mathscr{M}\subseteq\C^n$ be a complex manifold. Let $N$ be a complete
complex Hermitian manifold of dimension $k$. 
Let $\mathscr{F}=\left\{f_\alpha\right\}_{\alpha\in A}\subseteq \mathrm{Hol}(\mathscr{M},N)$.
The family $\mathscr{F}$ is not normal if and only if there exist (a) a compact set $K\subseteq\subseteq \mathscr{M}$, (b) a sequence $\{p_j\}\subseteq K$, (c) a sequence $\{f_j\}\subseteq \mathscr{F}$, (d) a sequence of reals $\{\rho_j\}$ satisfying $\rho_j>0$ and $\rho_j\longrightarrow 0^+$ and (e) a sequence of Euclidean vectors $\{\varepsilon_j\}\subseteq \C^n$ such that
\begin{align*}
g_j(\zeta)=f_j(p_j+\rho_j\varepsilon_j\zeta),\,\,\,\zeta\in \C,
\end{align*}
converges uniformly on compact subsets of $\C$ to a \emph{nonconstant} entire function $g$.
\end{lemma}
\section{Polynomial maps on matrix rings and power maps on general linear group}\label{sec:ring-poly}
We start with the following lemma, which reduces the problem of candidacy of $A$ in $\J_p(\M_n(\C))$ (equivalently $\F_p(\M_n(\C))$) to that of the Jordan canonical form of $A$.
\begin{lemma}\label{lem:red-jord}
    Let $p\in \C[z]$. For a matrix $X\in \M_n(\C)$, it is an element of the Julia set $\J_p(\M_n(\C))$ if and only if for any $Y\in\GL_n(\C)$, the matrix $YXY^{-1}$ is in the Julia set $\J_p(\M_n(\C))$.
\end{lemma}
\begin{proof}
    For any matrix $Y\in\GL_n(\C)$ and $p(z)=z^n+\sum\limits_{i=1}^{n}a_{n-i}z^{n-i}\in\C[z]$,
    \begin{align*}
        p(Y^{-1}XY)&=(Y^{-1}XY)^n+\sum\limits_{i=1}^na_{n-i}(Y^{-1}XY)^{n-i}\\
        &=Y^{-1}X^nY+\sum\limits_{i=1}^na_{n-i}Y^{-1}X^{n-i}Y=Y^{-1}XY.
    \end{align*}
    Since conjugation by $Y$ is a homeomorphism of the algebra $\M_n(\C)$, the result follows.
\end{proof}
Consider the map $w(x)=x^M$ for the polynomial algebra $\M_n(\C)$.
Let $g\in \M_n(\C)$ be regular semisimple; hence there exists a $Q\in\M_n(\C)$ such that $QgQ^{-1}=\diag(\alpha_1,\cdots,\alpha_n)$ with $\alpha_i\neq \alpha_j$ for all $i\neq j$.
Then $w^m(g)=g^{M^m}$; which is equal to $Q^{-1}\diag(\alpha_1^{M^m},\cdots,\alpha_n^{M^m})Q$.
Recall that the spectral radius of $g$ is by definition $\rho(g)=\max\limits_{1\leq i\leq n}|\alpha_i|$.
Let $g\in\J_w(\M_n(\C))$ such that $\rho(g)<1$, we get that $w^n(g)\longrightarrow 0$ as $n\longrightarrow\infty$.
Since $g$ is regular semisimple, and the set of regular semisimple elements forms an open set of $\M_n(\C)$, $w$ is a holomorphic map of full rank.
Hence, by \cite[Proposition 2.4]{FornaessSibony1998}, the Julia set $\J_w$ satisfies $w(\J_w) \subseteq \J_w$, which would imply that $0 \in \J_w$.

Let us calculate the Jacobian of $w$ at $g$. 
Note that up to conjugation $g = \operatorname{diag}(\alpha_1, \dots, \alpha_n)$.
Then for the standard matrix units $E_{ij}$, we have
$$
d(w^m)_g(E_{ij})
 = \left( \sum_{k=0}^{M^m-1} \alpha_i^{\,k} \alpha_j^{M^m-1-k} \right) E_{ij}.
$$
Hence each $E_{ij}$ is an eigenvector of $d(w^m)_g$, with corresponding eigenvalue
$$
\mu_{ij} =
\begin{cases}
\dfrac{\alpha_i^{M^m} - \alpha_j^{M^m}}{\alpha_i - \alpha_j}, & i \ne j, \\
M^m \alpha_i^{M^m-1}, & i = j.
\end{cases}
$$
Moreover, since $|\alpha_i|<1$, as $n\longrightarrow\infty$, the Jacobian tends to $0$, which implies that the zero matrix is an attracting point corresponding to the pair $(\M_n(\C),x^M)$.
This contradicts the assertion that $0 \in \J_w$.
Therefore, for any regular semisimple element $g$ with $\rho(g) < 1$, we have $g \not\in \J_w$.

Next, if $\rho(g) > 1$, the sequence ${w^m}$ diverges uniformly to infinity in a small neighborhood of $g$, since this neighborhood consists entirely of regular semisimple elements whose eigenvalues remain close to those of $g$. 
We note the whole discussion in the following remark.
\begin{remark}\label{rem:reg-sem}
For the pair $(\M_n(\C), x^M)$, the inclusion of a regular semisimple element $g \in \M_n(\C)$ in the Julia set implies that $\rho(g) = 1$.
\end{remark}

Let us point out that it is not true that for a matrix $A\in \M_n(\C)$ the sequence $\{p^m(A)\}$ is bounded if and only if the eigenvalues of $A$ are in $K_p(\C)$. 
Consider the power map $p(z)=z^2$ and the matrix $\begin{pmatrix}
    1& 1 & 0\\
    & 1 &1\\
    &&1
\end{pmatrix}\in \M_3(\C)$.
Obviously $1\in K_p(\C)$, but 
\begin{align*}
    p^m(A)=\begin{pmatrix}
        1 & 2^m & 2^{m-1}\\
        & 1 & 2^m\\
        &&1
    \end{pmatrix}
\end{align*}
implies that the sequence $\{p^m(A)\}_{m\geq 0}$ is unbounded.

Now we determine some elements $g\in\M_n(\C)$ such that $g\in K_{p}(\M_n(\C))$, where $p\in\C[x]$ is a monic polynomial. 
Recall the Jordan–Chevalley decomposition of a linear operator $g = g_s + g_u \in \M_n(\C)$, where $g_s$ is semisimple and $g_u$ is nilpotent. (In the literature, the nilpotent part is usually denoted by $g_n$; however, to avoid confusion with the dimension parameter $n$, we adopt the notation $g_u$.)
We have the following result;
\begin{proposition}\label{prop:filled-Julia}
    Let $\deg p\geq 2$ and $g\in\M_n(\C)$. The following statements hold.
    \begin{enumerate}
        \item if $g$ is semisimple and the eigenvalues of $g$ are in $K_p(\C)$, then $g\in K_p(\M_n(\C))$,
        \item if $g_u\neq 0$, and the eigenvalues of $g_s$ are in $\mathrm{Int}K_p(\C)$ (the interior of $K_p(\C)$), then $g\in K_p(\M_n(\C))$.
    \end{enumerate}
\end{proposition}
\begin{proof}
    (1) Let $g$ be semisimple. 
    Then there exists $Q\in \GL_n(\C)$ such that 
    $Q^{-1}gQ=\diag(\alpha_1,\cdots,\alpha_n)$ for some $\alpha_i\in\C$ for all $1\leq i\leq n$.
    Then 
    \begin{align*}
        p^m(g)=Q\begin{pmatrix}
            p^m(\alpha_1) & & &\\
            & p^m(\alpha_2) & &\\
            & & \ddots &\\
            &&&p^m(\alpha_n)
        \end{pmatrix}Q^{-1}.
    \end{align*}
    Since conjugation by an element of $\GL_n(\C)$ is a homeomorphism of $\M_n(\C)$, the sequence $\left\{p^m(g)\right\}_{m\geq0}$ is bounded if each of the sequence $\left\{p^m(\alpha_i)\right\}_{m\geq 0}$ for all $1\leq i\leq n$. 

(2) Since $\C$ is algebraically closed, there exists $Q\in\GL_n(\C)$ such that 
$Q^{-1}gQ=\bigoplus\limits_{i=1}^{\ell}J_{\alpha_i,m_i}$,
where $J_{\alpha_i,m_i}\in\M_{m_i}(\C)$ is the matrix
\begin{align*}
    \begin{pmatrix}
        \alpha_i & 1 &  & && \\
        &\alpha_i & 1 &   && \\
        &&\alpha_i  & && \\
         &  &  &\ddots && \\
         &&&&\alpha_i&1\\
        &  &  & && \alpha_i\\
    \end{pmatrix};
\end{align*}
here, $\alpha_i$s can be the same for two distinct $i$.
Whether this element belongs to $K_p(\M_n(\C))$ is equivalent to check whether the elements $J_{\alpha_i,m_i}$ are member of $K_p(\M_{m_i}(\C))$.
Without loss of generality, thus we may assume that $g$ is of the form $J_{\alpha,n}$. By induction, we have
\begin{align*}
    (p^m)(J_{\alpha,n})=\begin{pmatrix}
        (p^m)^{}(\alpha) & (p^m)^{(1)}(\alpha) & \cdots & (p^m)^{(n-1)}(\alpha)\\
        &(p^m)^{}(\alpha)  & \cdots & (p^m)^{(n-2)}(\alpha)\\
        &&\ddots&(p^m)^{(1)}(\alpha)\\
        &&&(p^m)^{}(\alpha)
    \end{pmatrix}.
\end{align*}
Given $\alpha\in \mathrm{Int} K_p(\C)$, it must belong to one bounded component of the Julia set which can be of three types, see \Cref{lem:suli-fatou-comp}. 
We study each case separately, starting with the case when $\alpha$ lies in an attracting basin.

\textbf{Case 1:}
Suppose $U$ is the basin of an attracting periodic cycle of period $N$, and let
$z\in U$. Then the sequence $\{p^m(z)\}$ is bounded by definition and we want to show that for every $k\geq 1$
the sequence $\{(p^m)^{(k)}(z)\}$ is bounded.

It is enough to prove the statement for the $p^N$ and the derivatives $(p^N)^{(m)}$, where $p^N(\alpha)=\alpha$; here $\alpha$ is an attracting fixed point of the map $q:=p^N$, $|q'(\alpha)|<1$.
By continuity of $q'$ at $\alpha$ we can choose $r>0$ and $0<\rho<1$ such that
\begin{align*}
\overline{D(\alpha,r)} \subset U \text{ and } \sup_{w\in D(\alpha,r)} |q'(w)| \le \rho <1.    
\end{align*}
which implies
\begin{align*}
|q(w)-q(\alpha)| = |q(w)-\alpha| \leq \rho\,|w-\alpha| \leq \rho r,    
\end{align*}
since $\overline{D(\alpha,r)}$ is compact.
Hence $q(D(\alpha,r))\subset D(\alpha,\rho r)\subset D(\alpha,r)$; on iteration one gets,
$q^m(D(\alpha,r))\subset D(\alpha,r)$ for all $m\geq0.$
Fix any $0<r_0<r$ and any point $y\in D(\alpha,r_0)$. Because $q^m$ maps
$\overline{D(\alpha,r)}$ into itself for every $m$, one gets the uniform bound
$
\sup\limits_{w\in D(\alpha,r)} |q^m(w)| \le M := |\alpha| + r,
$
independent of $m$.
Applying the Cauchy integral formula on the disk $D(\alpha,r)$: for any $k\geq 1$ and any
$y\in D(\alpha,r_0)$ with $r_0<r$,
\begin{align*}
(q^m)^{(k)}(y) = \frac{k!}{2\pi i}\int\limits_{|\zeta-\alpha|=r}\frac{q^m(\zeta)}{(\zeta-y)^{k+1}}\,d\zeta. 
\end{align*}
Using $|\zeta-y|\ge r-r_0$ on the circle, we get that
\begin{align*}
\big| (q^m)^{(k)}(y)\big|
\leq \frac{k!}{(r-r_0)^k}\,\sup_{|\zeta-\alpha|=r}|q^m(\zeta)|
\leq \frac{k!\,M}{(r-r_0)^k},    
\end{align*}
which is independent of $m$. Thus for every fixed $k$ and every $y\in D(\alpha,r_0)$
the sequence $\{(q^m)^{(k)}(y)\}_{m\ge0}$ is uniformly bounded.
For arbitrary $\ell$, we may find $r$ and $\ell'$ such that $\ell= Nr+\ell'$ with $0\leq \ell'<N$. Then one observes that $p^\ell=p^{Nr}\circ p^{\ell'}$ and applies the chain rule.
\textbf{Case 2:} Let $\alpha\in U$ where $U$ is an immediate attracting petal for a parabolic
periodic point of $p$. 
As in the previous case, we may assume that $\alpha$ is a fixed point, keeping the notation $q=p^N$.
The iterates $q^m$ converge uniformly on compact subsets of the
immediate petal $U$ to the parabolic fixed point $\alpha$.
Fix the given $z\in U$ and choose
$r>0$ with $\overline{D(z,r)}\subseteq U$. Uniform convergence on the compact set
$\overline{D(z,r)}$ implies the family $\{q^m\}_{m\ge0}$ is uniformly bounded there, and hence there
exists $M>0$ such that
\begin{align*}
\sup_{\substack{m\geq0\\ \zeta\in\overline{D(z,r)}}} |q^m(\zeta)| \le M.
\end{align*}
Now, for every $k\geq 1$ and every
$m\geq 0$,
\begin{align*}
(q^m)^{(k)}(z)
&= \frac{k!}{2\pi i}\int\limits_{|\zeta-z|=r}\frac{q^m(\zeta)}{(\zeta-z)^{k+1}}\,d\zeta,
\end{align*}
and therefore,
\begin{align*}
\big|(q^m)^{(k)}(z)\big|
&\le \frac{k!}{r^k}\,\sup_{|\zeta-z|=r}|q^m(\zeta)|
\le \frac{k!\,M}{r^k},
\end{align*}
which is independent of $m$. Thus for each fixed $k$ the sequence $\{(q^m)^{(k)}(z)\}$ is
uniformly bounded. 



\textbf{Case 3:}
Lastly let $\alpha\in U$ such that $U$ is a Siegel disk for $p$, i.e.\ there
exists a conformal map $\phi:\D_r\to U$ from some disk $\D_r=\{w\in\C:|w|<r\}$ and an irrational $\theta$ such that
\begin{align*}
\phi(e^{2\pi i\theta}w)=p(\phi(w))\qquad\text{for }w\in\D_R.
\end{align*}
By the conjugacy relation $\phi\circ R=\;p\circ\phi$ with $R(w)=e^{2\pi i\theta}w$ we have
for every $m\ge0$
\begin{align*}
p^m=\phi\circ R^m\circ\phi^{-1}.
\end{align*}
Set $w_m:=R^m(w_0)=e^{2\pi i m\theta}w_0$. Since $|w_m|=|w_0|<r$ for all $m$, the points
$\{w_m\}$ lie in a compact subset of $\D_r$, and hence the maps $\phi$ and its inverse $\phi^{-1}$ have all their derivatives to be uniformly bounded on the relevant compact sets.

For the first derivative,
\begin{align*}
(p^m)'(z)
&= \phi'(w_m)\cdot (R^m)'(w_0)\cdot (\phi^{-1})'(z)
= \phi'(w_m)\,e^{2\pi i m\theta}\,(\phi^{-1})'(z).
\end{align*}
Hence
\begin{align*}
\big|(p^m)'(z)\big| = |\phi'(w_m)|\;|(\phi^{-1})'(z)|,
\end{align*}
and the right-hand side is uniformly bounded in $m$ because $\{\phi'(w_m)\}_m$ is bounded
(on the compact set containing the $w_m$) while $(\phi^{-1})'(z)$ is a fixed finite number.

For higher derivatives, firstly note that $R^m(z)=e^{2\pi m i}z$, and hence all the higher ($\geq 2$) derivatives vanish. 
Further, for each fixed $k$, there exist
constants $C_{j,\ell}$ (depending only on $k$) such that
\begin{align*}
(p^m)^{(k)}(z)=\sum_{1\leq j\leq k,\;0\leq\ell\leq k} C_{j,\ell}\; \phi^{(j)}(w_m)\;(\phi^{-1})^{(\ell)}(z),
\end{align*}
by the chain rule.
Because $w_m$ remains in a compact subset of $\D_r$, the values $\phi^{(j)}(w_m)$ are
uniformly bounded in $m$ for every fixed $j$, and the finitely many derivatives
$(\phi^{-1})^{(\ell)}(z)$ are fixed numbers. Therefore each $(p^m)^{(k)}(z)$ is uniformly
bounded in $m$.
\end{proof}
Clearly $\overline{K_p}=\left\{g\in \M_n(\C):\text{all eigenvalues of $g$ are in $K_p(\C)$}\right\}$.
Note that if at least one eigenvalue of a matrix $g\in \M_n(\C)$ lies in $\C\setminus K_p(\C)$, then the sequence $\left\{p^n(g)\right\}$ diverges uniformly in a neighbourhood $U$ of $g$. 
Thus $\J_p(\M_n(\C))\subseteq \overline{K_p(\M_n(\C))}$.
Before proving the final theorem, we state the following lemma, which constructs an open set around a given matrix, inside $\M_n(\C)$, which will be needed in the proof.
\begin{lemma}\label{lem:nbd-mat}
Let $g\in\M_n(\C)$, and $\alpha_1,\cdots,\alpha_n$ be eigenvalues of $g$ (counted with multiplicity). Let $U$ be an open set of $\C$, containing $\alpha_i$ for all $i$. Then the set
\begin{align*}
U(g)=\left\{h\in \M_n(\C):\sigma(x)\subseteq U\right\}
\end{align*}
is an open set.
\end{lemma}
\begin{proof}
    See \cite{stackexchange}.
\end{proof}
Now we are ready to prove the main result of the article.
\theoremone*
\begin{proof}
Let $\sigma(X)\subseteq \F_p(\C)$.
We first assume that none of the eigenvalues are in the basin of infinity, and let the eigenvalues be $\lambda_1,\lambda_2,\cdots,\lambda_n$ counted with multiplicity.

Since the connected Fatou components (of $\F_p(\C)$) are open, letting $F_i$ to be the Fatou component containing $\alpha_i$ in $\F_p(\C)$, one gets that $\bigcup\limits_{i=1}^n F_i$ is an open set.
Hence by using \Cref{lem:nbd-mat} the set
\begin{align*}
\mathscr U=\left\{Y\in\M_n(\C):\sigma(Y)\subseteq \bigcup\limits_{i=1}^n F_i\right\}
\end{align*}
is open in $\M_n(\C)$.
We will first show that for all $X_0\in \mathscr U$, there exists an open set $U_0$ containing $X_0$ such that the family $\left\{p^m\right\}$ is a 
uniformly bounded family on $U_{0}$; thus proving that on $\mathscr U$, the family $\{p^m\}$ is locally uniformly bounded.
After that, we will apply \Cref{lem:normal-GK} to conclude our result.

Given $\sigma(X_0)$ is finite.
For $\lambda_i\in \sigma(X_0)$, let $F_i$ be the bounded Fatou component containing $\lambda_i$. 
Choose $C_i$ to be a small circle around $\lambda_i$, contained in $F_i$, for all $i=1,\cdots,n$.
Choose $\Gamma=\bigcup\limits_{i=1}^n C_i$, which is a compact subset of the Fatou set $\F_p(\C)$.
Since $p^k$ is normal on the Fatou components, the family $\{p^m\}$ is uniformly bounded on compact sets, and more precisely on $\Gamma$.
Hence there exists a constant $C$, such that
\begin{align*}
\sup\limits_{\substack{m\geq 0\\\gamma\in\Gamma}}|p^m(\gamma)|\leq C<\infty.
\end{align*} 
Since $\Gamma\cap\sigma(X_0)=\emptyset$, one has $(\gamma I-X_0)^{-1}$ is well defined, where $I$ is the $n\times n$ identity matrix.
Since $\Gamma$ is compact, there exists $M<\infty$ such that
\begin{align*}
M=\sup\limits_{\gamma\in\Gamma}||(\gamma I - X_0)^{-1}||.
\end{align*} 
Let $$U_0=\left\{Y\in \M_n(\C):||Y-X_0||<\dfrac{1}{2M}\right\}\cap\mathscr U,$$
which further implies that
for every $Z\in U_0$, and $\gamma\in \Gamma$,
\begin{align*}
||(\gamma I - X_0)^{-1}(Z-X_0)||\leq M||Z-X_0||<1.
\end{align*}
 This shows that $1$ is not an eigenvalue of $(\gamma I - X_0)^{-1}(Z-X_0)$; hence $I-(\gamma I - X_0)^{-1}(Z-X_0)$ is invertible.
Thus $\gamma I-Z=(\gamma I-X_0)(I-(\gamma I - X_0)^{-1}(Z-X_0))$ is invertible for all $\gamma\in\Gamma$.
Observe that for all $Z\in U$, and $\gamma\in \Gamma$ one has $||(\gamma I-Z)^{-1}||\leq\dfrac{M}{1-M\cdot (1/2M)}$, since $||(\gamma I - X_0)^{-1}(Z-X_0)||\leq M||Z-X_0||<1$; hence we have that
\begin{align*}
\sup\limits_{\substack{Z\in U_0\\\gamma\in\Gamma}}||(\gamma I- Z)||<2M.
\end{align*}
Using the Cauchy integral formula for the matrix valued functions \cite[Theorem 1.9, and page 190]{Dunford1943}, (see also \cite{Higmanbook}) for any $m$ and $Z\in U_0$,
\begin{align*}
p^m(Z)
=\dfrac{1}{2\pi i}\int\limits_{\Gamma} p^m(\zeta)(\zeta I - Z)^{-1}d\zeta.
\end{align*}
Then for all $Z\in U_0$ and $m\geq 0$ (let $\ell(\Gamma)$ denote the length of $\Gamma$),
\begin{align*}
||p^m(Z)||&=\left|\left|\dfrac{1}{2\pi i}\int\limits_{\Gamma} p^m(\zeta)(\zeta I - Z)^{-1}d\zeta\right|\right|\\
&\leq \dfrac{1}{2\pi}\ell(\Gamma)\cdot C\cdot 2M.
\end{align*}
This proves that on the open set $\mathscr{U}$ the family is locally uniformly bounded.

Now, if possible, let the family not be normal on $\mathscr U$. 
Then by \Cref{lem:normal-GK}, there exists a compact  $K\subseteq\subseteq \mathscr U$, a sequence $\{p_j\}\subseteq K$, $\{f_j=p^{m_j}\}\subseteq\mathscr{P}=\{p^m\}_{m\geq 0}$, positive reals $\rho_j\longrightarrow 0$, and vectors $\epsilon_j\in\M_n(\C)$ such that the sequence of maps $\{g_j\}$ defined as
\begin{align*}
g_j(\zeta)=f_j(p_j+\rho_j\epsilon_j\zeta),
\end{align*}
converges uniformly on compact subset of $\C$ to a non-constant entire function $g$ as $j\longrightarrow\infty$.

Since $K$ is compact and $\{p^k\}$ is locally uniformly bounded on $\mathscr U$ there exists an open set $V\subseteq \mathscr U$ such that $K\subseteq V$ with the property that
\begin{align*}
\sup\limits_{\substack{m\geq 0\\X\in V}}||p^m(X)||<C<\infty.
\end{align*}
Let $B_R=\{z\in\C:||z||\leq R\}$. Since $K\subseteq V$, we have that $d=\mathrm{dist}(K,\partial V)>0$. For each $\zeta\in B_R$ one has $||p_j+\rho_j\epsilon_j\zeta-p_j||\leq \rho_j|\epsilon_j|R$. Let $J$ be such that for all $j\geq J$ we have $\rho_j|\epsilon_j|R<d$ (this is possible as $\rho_j\longrightarrow 0$ and $|\epsilon_j|=1$.).

Then for $j\geq J$ and $\zeta\in B_R$, we have that $p_j+\rho_j\epsilon_j\zeta\in V$, whence
\begin{align*}
||g_j(\zeta)||=||p^{k_j}(p_j+\rho_j\epsilon_j\zeta)||\leq C,
\end{align*}
which shows that $g$ must be a constant function. Hence the family $\{p^m\}_{m\geq 1}$ is normal on the open set $U_0$. 
If $\sigma(X)\subseteq \F_p(\C)$ contains an element in the basin of infinity, then $||p^m(X)||$ converges to infinity uniformly. Thus we conclude that $\sigma(X)\subseteq \F_p(\C)$ implies that $X\in\F_p(\C)$.

Conversely let $\sigma(X)\not\subseteq\F_p(\C)$ and let $\alpha\in\J_p(\C)\cap\sigma(X)$. Then, in the Jordan canonical form of $X$, there exists a block of the form
\begin{align*}
    \begin{pmatrix}
        \alpha &1&&&&\\
        &\alpha&1&&&\\
        &&\ddots&\\
        &&&\alpha&1\\
        &&&&\alpha
    \end{pmatrix}.
\end{align*}
Note that for an open set $\emptyset\neq W\subseteq \C$ containing $\alpha$ such that $\{p|_W^m\}$ is not normal therein. 
Then $\{p^m\}$ is not normal on the open set $\mathscr W=\{Y\in\M_n(\C):\sigma(X)\subseteq W\}$.
Since for any open set $\mathscr W'\subseteq \M_n(\C)$, containing $X$, the intersection $\mathscr W\cap\mathscr W'$ is non-empty, the family $\{p^m\}$ is not normal on any neighborhood of $X$. This finishes the proof for the first part of the theorem.

To prove the next statement, note that we consider these maps as $w:\GL_n(\C)\longrightarrow\M_n(C)$. Thus \Cref{lem:normal-GK} is applicable. The rest of the proof is a replica of the above proof, since $\GL_n(\C)$ is open in $\M_n(\C)$.
\end{proof}
The following corollary is immediate.
\begin{corollary}\label{cor:Fatou-polynomial}
    Let $p\in\C[z]$ be a monic polynomial of degree $\geq 2$. Then
    \begin{align*}
        \J_p(\M_n(\C))&=\partial \overline{K_p(\M_n(\C))}\\
        &=\{X\in \M_n(\C):\sigma(X)\cap\J_p(\C)\neq\emptyset\}.
    \end{align*}
\end{corollary}
\begin{corollary}\label{cor:no-wander}
    There exists no wandering Fatou component of the pair $(p,\M_n(\C))$ where $p\in\C[z]$ is a monic polynomial of degree $\geq 2$.
\end{corollary}
\begin{proof}
    Let $C$ be a Fatou component of the pair $(p,\M_n(\C))$.
    To prove our claim, we must show that there exists $m$, a positive integer, such that $p^m(C)\subseteq C$.

    Let $X\in C$, since the centralizer $\mathscr Z_{\GL_n(\C)}(X)\neq\emptyset$, and $\{gYg^{-1}:Y\in C,g\in\GL_n(\C)\}$ is connected we must have that $gXg^{-1}\in C$ for all $g\in \GL_n(\C)$.
The map
\begin{align*}
\chi : \M_n(\C) &\longrightarrow \C^{n}, \\
A &\longmapsto (a_1(A),a_2(A),\dots,a_n(A)),
\end{align*}
assigning to each matrix $A$ the vector of coefficients of its characteristic polynomial
\begin{align*}
\det(tI - A) = t^n + a_1(A)t^{n-1} + \cdots + a_n(A),
\end{align*}
is a polynomial map in the entries of $A$, and hence continuous.
Moreover, it is well known that the roots of a complex polynomial depend continuously on its coefficients (see, for instance, \cite{HarrisMartin1987}).
Hence, we can construct 
\begin{align*}
    \lambda_1:C\longrightarrow\C,\cdots,\lambda_n:C\longrightarrow \C,
\end{align*}
such that $\sigma(X)=\{\lambda_i(X):1\leq i\leq n\}$.
From \Cref{thm:FJ-polynomial-map}, it follows that an element lies in $\F_p(\M_n(\C))$ if and only if all eigenvalues lie in $\F_p(\C)$.
Hence for all $1\leq i\leq n$, one has that $\lambda_i(C)\subseteq C_i$, where $C_i$ is a Fatou component of $\F_p(\C)$.
Since there exists no Fatou component of the pair $(p,\C)$, one must have the existence of $m$ such that $p^m(C_i)\subseteq C_i$ for all $1\leq i\leq n$.
Hence $p^m(C)\subseteq C$.
\end{proof}
\printbibliography
\vspace{2em}
\end{document}